\documentclass[11pt]{amsart}

\usepackage{amsmath,amsthm,amssymb,mathtools}
\usepackage[utf8]{inputenc}
\usepackage[T1]{fontenc}
\usepackage{microtype, a4wide, mathrsfs}
\usepackage[shortlabels]{enumitem}
\usepackage{hyperref,dsfont}

\newtheorem{theorem}{Theorem}[section]
\newtheorem{lemma}[theorem]{Lemma}
\newtheorem{proposition}[theorem]{Proposition}
\newtheorem{corollary}[theorem]{Corollary}

\theoremstyle{definition}
\newtheorem{definition}[theorem]{Definition}
\newtheorem{remark}[theorem]{Remark}
\newtheorem{question}[theorem]{Question}
\newcommand{\R}{\mathbb{R}}
\newcommand{\T}{\mathbb{T}}
\newcommand{\N}{\mathbb{N}}
\newcommand{\uo}{\ensuremath{\mathsf{uo}}}

\newcommand{\Buo}{\ensuremath{\mathsf{Buo}}}
\renewcommand{\le}{\leqslant}
\renewcommand{\ge}{\geqslant}

\title[Bourgain-uo sequential completeness in vector lattices]{Bourgain--uo sequential completeness in vector lattices}
\author[T.~Kania]{Tomasz Kania}
\address[T.~Kania]{Mathematical Institute\\Czech Academy of Sciences\\\v Zitn\'a 25 \\115 67 Praha 1\\Czech Republic  and  Institute of Mathematics and Computer Science\\ Jagiellonian University\\ {\L}ojasiewicza 6, 30-348 Krak\'{o}w, Poland
}
\email{kania@math.cas.cz, tomasz.marcin.kania@gmail.com}

\author[J.~Swaczyna]{Jaros{\l}aw Swaczyna}
\address[J.~Swaczyna]{Institute of Mathematics, {\L}\'od\'z University of Technology, Aleje Politechniki 8,  93-590 {\L}\'od\'z, Poland}
\email{jaroslaw.swaczyna@p.lodz.pl}
\date{\today}
\thanks{Institute of Mathematics, Czech Academy of Sciences; RVO: 67985840. The second-named author acknowledges with thanks funding received from NCN Sonata-Bis 13 (2023/50/E/ST1/00067). }
\dedicatory{In memoriam: Jan Mikusi\'{n}ski (1913--1987)}

\begin{document}
\begin{abstract}
We revisit Bourgain's 1981 counterexample to the sequential completeness of the
`pointwise plus domination' convergence on $\ell_1$ from the perspective of
vector lattices. In this setting, we show that for sequences the associated
notion of Bourgain--uo convergence coincides with ordinary order convergence.
Motivated by Bourgain's construction, we introduce a strengthened,
subsequence-invariant notion of Cauchy sequence: a sequence $(x_n)$ in a vector
lattice $E$ is called \Buo-\emph{Cauchy} if for every strictly increasing
sequence $(n_k)$ the differences $x_{n_{k+1}}-x_{n_k}$ converge to $0$ in order
in $E$.

We first show that sequential \Buo-completeness forces $\sigma$-order
completeness. Thus every non-$\sigma$-order complete vector lattice fails
sequential \Buo-completeness. In particular, free Banach lattices
$\mathrm{FBL}(E)$ are not sequentially \Buo-complete whenever $\dim E>1$.

On the positive side, we prove that the classical sequence lattices $c_0$ and
$\ell_\infty$ are sequentially \Buo-complete: every \Buo-Cauchy sequence
converges in order, and hence in the \Buo\ sense.

Finally, we obtain a sharp metric characterisation for bounded Lipschitz
function lattices: the vector lattice $\mathrm{Lip}_b(X)$ of bounded Lipschitz
functions on a metric space $(X,d)$ is sequentially \Buo-complete if and only
if $X$ is uniformly discrete.
\end{abstract}

% 46B42  Banach lattices
% 46E30  Spaces of measurable functions (Lp-spaces, Orlicz spaces, ...)

\maketitle
\section{Introduction}

In \cite{Bourgain81}, J.~Bourgain studied on $\ell_1$ the convergence obtained by
combining coordinatewise convergence with domination by a single element of
$\ell_1$ and proved that this convergence is not sequentially complete.  This
answered a question of Jan Mikusi\'{n}ski and suggests a natural order-theoretic
problem in the setting of vector lattices.

If $E$ is a vector lattice, the natural analogue of coordinatewise or
almost-everywhere convergence is unbounded order ($\uo$) convergence, while the
domination requirement becomes order boundedness by a single positive vector.
This leads to the notion of Bourgain--uo convergence introduced below.  For
sequences, however, the convergence notion itself turns out not to be new:
Bourgain--uo convergence coincides with ordinary order convergence
(Proposition~\ref{prop:Buo=order}).  The real novelty lies in the associated
Cauchy condition.

Motivated by Bourgain's finite-block argument, we therefore introduce a
subsequence-invariant Cauchy notion: a sequence $(x_n)$ in a vector lattice $E$
is called \Buo-Cauchy if for every strictly increasing sequence $(n_k)$ the
differences $x_{n_{k+1}}-x_{n_k}$ converge to $0$ in order.  This condition is
weaker than the usual order-Cauchy condition, but it is still strong enough to
force substantial order-theoretic structure.

The main results of the paper are these.
\begin{enumerate}[label=(\roman*)]
\item If every \Buo-Cauchy sequence in $E$ converges in order, then $E$ is
$\sigma$-order complete.  In particular, $\mathrm{FBL}(E)$ fails sequential
\Buo-completeness whenever $\dim E>1$.
\item The classical sequence lattices $c_0$ and $\ell_\infty$ are sequentially
\Buo-complete.
\item For a metric space $(X,d)$, the lattice $\mathrm{Lip}_b(X)$ is
sequentially \Buo-complete if and only if $X$ is uniformly discrete.
\end{enumerate}

The paper is organised as follows.  Section~\ref{sec:preliminaries} fixes
notation, records our conventions on order convergence, and compares
\Buo-Cauchy sequences with the usual order-Cauchy notion.
Section~\ref{sec:general-obstruction} contains the general obstruction coming
from $\sigma$-order completeness and treats the sequence spaces $c_0$ and
$\ell_\infty$.  Section~\ref{sec:LipX} deals with bounded Lipschitz lattices,
and Section~\ref{sec:open-problems} collects several open questions.

\section{Preliminaries}\label{sec:preliminaries}

Throughout, $E$ denotes a vector lattice.

\begin{definition}[Order convergence]\label{def:order_convergence}
Let $E$ be a vector lattice and $(x_\alpha)$ a net in $E$.
We write $x_\alpha\xrightarrow{\mathrm{o}}x$ if there exists a net $(u_\beta)$ in $E_+$
such that $u_\beta\downarrow 0$ and for every $\beta$ there exists $\alpha_0$ with
\[
|x_\alpha-x|\le u_\beta \qquad(\alpha\ge \alpha_0).
\]
Here $u_\beta\downarrow 0$ means that $(u_\beta)$ is decreasing and $\inf_\beta u_\beta=0$ in $E$.
For sequences $(x_n)$ we use the same definition (with $\alpha=n\in\N$).
\end{definition}

\begin{remark}\label{rem:sigma_order_convention}
Some authors use a sequential variant of order convergence for sequences, often
denoted $x_n\xrightarrow{\sigma\text{-}\mathrm{o}}x$, in which the dominating
family can be chosen to be a sequence.  We will not need a detailed comparison
of these two notions below.  They agree in many standard settings, for example
in almost $\sigma$-order complete vector lattices; see
\cite{AbramovichSirotkin2005,TaylorThesis2018} for background.  Throughout,
``order convergence'' always means Definition~\ref{def:order_convergence}.
\end{remark}

\begin{definition}[Unbounded order convergence]\label{def:uo_convergence}
Let $E$ be a vector lattice and $(x_\alpha)$ a net in $E$.  We write
$x_\alpha\xrightarrow{\uo}x$ if
\[
|x_\alpha-x|\wedge u\xrightarrow{\mathrm{o}}0
\qquad\text{for every }u\in E_+.
\]
For sequences we use the same definition.
\end{definition}

\begin{remark}\label{rem:uo_not_pointwise}
In Dedekind complete function lattices such as $\R^X$ or $L_0(\mu)$,
$\uo$-convergence agrees with pointwise (respectively a.e.) convergence; see
\cite[{\S}3]{GaoXanthos2014}.  In proper sublattices such as $C(X)$ or
$\mathrm{Lip}_b(X)$, $\uo$-convergence---and even order convergence---can be
strictly weaker than pointwise convergence.  For example, in $C([0,1])$ the
functions $f_n(t):=(1-nt)_+$ satisfy $f_n\downarrow 0$ in order although
$f_n(0)=1$ for all $n$.
\end{remark}

\begin{definition}[Bourgain--uo convergence]
A sequence $(x_n)\subset E$ \emph{\Buo-converges} to $x\in E$ if
$x_n\xrightarrow{\uo}x$ and there exists $y\in E_+$ such that
$|x_n|\le y$ for all $n$.
\end{definition}

The next proposition shows that, for sequences, Bourgain--uo convergence is
exactly order convergence.  We are grateful to Vladimir Troitsky and Mitchell
Taylor for this observation.

\begin{proposition}\label{prop:Buo=order}
For a sequence $(x_n)$ in a vector lattice $E$ and $x\in E$ the following are equivalent:
\begin{enumerate}[label=(\roman*)]
\item $x_n\xrightarrow{\Buo} x$;
\item $x_n\xrightarrow{\mathrm{o}} x$.
\end{enumerate}
\end{proposition}
\begin{proof}
(i)$\Rightarrow$(ii):
Assume $x_n\xrightarrow{\uo}x$ and $|x_n|\le y$ for some $y\in E_+$.
Set $w:=(|x|-y)^+\in E_+$. Since $|x_n|\le y$ we have
\[
|x|-y \le |x|-|x_n|,
\]
hence
\[
w=(|x|-y)^+ \le (|x|-|x_n|)^+.
\]
By the reverse triangle inequality, $|x|-|x_n|\le |x-x_n|$, so $(|x|-|x_n|)^+\le |x-x_n|$.
Therefore $w\le |x-x_n|$, and consequently
\[
|x-x_n|\wedge w = w \qquad(n\in\N).
\]
Taking $u=w$ in the definition of $\uo$-convergence yields
$w=|x-x_n|\wedge w\xrightarrow{\mathrm{o}}0$, hence $w=0$.
Thus $|x|\le y$, and for all $n$,
\[
|x_n-x|\le |x_n|+|x|\le 2y.
\]
Since $\uo$-convergence gives $|x_n-x|\wedge (2y)\xrightarrow{\mathrm{o}}0$ and
$|x_n-x|=|x_n-x|\wedge (2y)$, we get $|x_n-x|\xrightarrow{\mathrm{o}}0$,
i.e.\ $x_n\xrightarrow{\mathrm{o}}x$.

(ii)$\Rightarrow$(i):
Assume $x_n\xrightarrow{\mathrm{o}}x$ in the sense of Definition~\ref{def:order_convergence}.
Then there exist a net $(z_\beta)\downarrow 0$ in $E_+$ and integers $N_\beta$ such that
$|x_n-x|\le z_\beta$ whenever $n\ge N_\beta$.
Fix $\beta_0$. For $n\ge N_{\beta_0}$ we have
\[
|x_n|\le |x|+|x_n-x|\le |x|+z_{\beta_0}.
\]
Define
\[
y:=\bigl(|x|+z_{\beta_0}\bigr)\vee |x_1|\vee\cdots\vee |x_{N_{\beta_0}-1}|\in E_+.
\]
Then $|x_n|\le y$ for all $n$.
Moreover, for every $u\in E_+$ and $n\ge N_\beta$,
\[
|x_n-x|\wedge u \le z_\beta\wedge u,
\]
and since $(z_\beta\wedge u)\downarrow 0$ we obtain $|x_n-x|\wedge u\xrightarrow{\mathrm{o}}0$.
Hence $x_n\xrightarrow{\uo}x$, and together with $|x_n|\le y$ this implies $x_n\xrightarrow{\Buo}x$.
\end{proof}

\begin{definition}[Order-Cauchy]\label{def:order_cauchy}
A sequence $(x_n)$ in a vector lattice $E$ is called \emph{order Cauchy} if the net
\[
(x_n-x_m)_{(n,m)\in\N^2}
\]
converges to $0$ in order (where $\N^2$ is directed by the coordinatewise order).
Equivalently, there exists a net $(u_\beta)\downarrow 0$ in $E_+$ such that for every $\beta$
there is $N$ with
\[
|x_n-x_m|\le u_\beta\qquad(n,m\ge N).
\]
\end{definition}

\begin{definition}\label{def:buo_cauchy}
A sequence $(x_n)$ in a vector lattice $E$ is called \Buo-\emph{Cauchy} if for
every strictly increasing sequence $(n_k)$ the differences
$x_{n_{k+1}}-x_{n_k}$ converge to $0$ in order in $E$.
\end{definition}

\begin{lemma}\label{lem:order_cauchy_implies_Buo_cauchy}
Every order-Cauchy sequence is \Buo-Cauchy.
\end{lemma}

\begin{proof}
Let $(x_n)$ be order Cauchy and let $(n_k)$ be strictly increasing.
Put $d_k:=x_{n_{k+1}}-x_{n_k}$.
By Definition~\ref{def:order_cauchy} there exists a net $(u_\beta)\downarrow 0$ in $E_+$ such that
for each $\beta$ there is $N$ with $|x_n-x_m|\le u_\beta$ whenever $n,m\ge N$.
Fix $\beta$ and choose $k_0$ with $n_{k_0}\ge N$. Then $|d_k|\le u_\beta$ for all $k\ge k_0$,
so $d_k\xrightarrow{\mathrm{o}}0$. Since $(n_k)$ was arbitrary, $(x_n)$ is \Buo-Cauchy.
\end{proof}

\begin{remark}\label{rem:BuoCauchy_vs_orderCauchy}
Lemma~\ref{lem:order_cauchy_implies_Buo_cauchy} shows that \Buo-Cauchyness is weaker
than the usual order-Cauchy condition.  The point of \Buo-Cauchyness is its subsequence
invariance and its connection with Bourgain's finite-block phenomenon \cite{Bourgain81}.
\end{remark}

\begin{lemma}\label{lem:disjoint_order_null}
Let $E$ be a vector lattice. If $(x_n)$ is a disjoint sequence in $E$ and $\{x_n:n\in\N\}$ is order bounded,
then $x_n\xrightarrow{\mathrm{o}}0$.
\end{lemma}

\begin{proof}
Disjoint sequences are $\uo$-null; see, for instance, \cite[{\S}3]{GaoXanthos2014}.
If $|x_n|\le u$ for some $u\in E_+$, then $|x_n|=|x_n|\wedge u$ for all $n$, so $\uo$-nullness yields
$|x_n|\xrightarrow{\mathrm{o}}0$, i.e.\ $x_n\xrightarrow{\mathrm{o}}0$.
\end{proof}

\begin{lemma}\label{lem:monotone_order_limit}
Let $E$ be a vector lattice and let $(x_n)$ be an increasing sequence in $E$.
If $x_n\xrightarrow{\mathrm{o}}x$, then $x=\sup_n x_n$ in $E$.
\end{lemma}

\begin{proof}
Let $(u_\beta)\downarrow 0$ witness $x_n\xrightarrow{\mathrm{o}}x$.
First we show that $x$ is an upper bound of $\{x_n\}$.
If $x_{n_0}>x$ for some $n_0$, then $w:=(x_{n_0}-x)^+>0$ and $x_n-x\ge w$ for all $n\ge n_0$.
Choose $\beta$ such that $w\not\le u_\beta$ (possible since $u_\beta\downarrow 0$), contradicting
$x_n-x\le u_\beta$ eventually.

Now let $y$ be any upper bound of $\{x_n\}$. Then $0\le (x-y)^+\le x-x_n$ for all $n$.
If $(x-y)^+>0$, the same argument contradicts $x-x_n\xrightarrow{\mathrm{o}}0$.
Hence $(x-y)^+=0$, i.e.\ $x\le y$. Thus $x$ is the least upper bound.
\end{proof}

The next observation explains how sequential \Buo-completeness fits into the
usual hierarchy of order-theoretic completeness notions.

\begin{proposition}\label{prop:Buo_complete_implies_sigma_order_complete}
If every \Buo-Cauchy sequence in a vector lattice $E$ converges in order, then
$E$ is $\sigma$-order complete.
\end{proposition}

\begin{proof}
By Lemma~\ref{lem:order_cauchy_implies_Buo_cauchy}, every order-Cauchy
sequence in $E$ is \Buo-Cauchy and hence order convergent. Thus $E$ is
order-Cauchy complete for sequences. It is standard that a vector lattice is
$\sigma$-order complete if and only if every order-Cauchy sequence is order
convergent; see \cite[Proposition~18.46 and Remark~18.47]{TaylorThesis2018}.
Therefore $E$ is $\sigma$-order complete.
\end{proof}

\begin{corollary}\label{cor:non_sigma_order_complete_not_Buo_complete}
If a vector lattice $E$ is not $\sigma$-order complete, then $E$ is not
sequentially \Buo-complete.
\end{corollary}

\begin{proof}
Immediate from Proposition~\ref{prop:Buo_complete_implies_sigma_order_complete}.
\end{proof}

\begin{corollary}\label{cor:FBL_not_Buo_complete}
Let $E$ be a Banach space with $\dim E>1$. Then $\mathrm{FBL}(E)$ is not
sequentially \Buo-complete.
\end{corollary}

\begin{proof}
By \cite[Proposition~2.11(3)]{OikhbergTaylorTradaceteTroitsky2024},
$\mathrm{FBL}(E)$ is not $\sigma$-order complete whenever $\dim E>1$.
Apply Corollary~\ref{cor:non_sigma_order_complete_not_Buo_complete}.
\end{proof}

\section{A general obstruction and sequence spaces}\label{sec:general-obstruction}

We begin with the general obstruction supplied by
Proposition~\ref{prop:Buo_complete_implies_sigma_order_complete} and then turn
to the sequence spaces $c_0$ and $\ell_\infty$.

\subsection{$c_0$ and $\ell_\infty$}\label{sec:c0-linfty}

In this subsection we study \Buo-Cauchy sequences in the classical sequence
lattices $c_0$ and $\ell_\infty$.  Recall that in these spaces
$\uo$-convergence is just coordinatewise convergence of sequences (see, for
instance, \cite[{\S}3]{GaoXanthos2014}).

Let $(x_n)\subset c_0$ (or $(x_n)\subset\ell_\infty$) be \Buo-Cauchy.  Fix a
coordinate $j\in\N$ and consider the scalar sequence $(x_n(j))_n$ in $\R$.
The definition of \Buo-Cauchy in the one-dimensional lattice $\R$ implies
that $(x_n(j))_n$ is a Cauchy sequence in the usual sense (indeed, for
every subsequence $(n_k)$ the differences $x_{n_{k+1}}(j)-x_{n_k}(j)$
converge to $0$ in $\R$). Hence $(x_n(j))_n$ converges in $\R$ for every
$j$, and so $(x_n)$ admits a pointwise limit $x\in\R^\N$.

The next lemma shows that this limit already belongs to the ambient Banach
lattice.

\begin{lemma}\label{lem:pointwise_limit_in_c0}
Let $(x_n)\subset c_0$ (respectively $(x_n)\subset\ell_\infty$) be
\Buo-Cauchy and let $x\in\R^\N$ denote its pointwise limit. Then
$x\in c_0$ (respectively $x\in\ell_\infty$).
\end{lemma}

\begin{proof}
We treat the cases of $c_0$ and $\ell_\infty$ separately.

\emph{Case $c_0$.}
Suppose, for a contradiction, that $x\notin c_0$. Then there exist
$\varepsilon>0$ and a strictly increasing sequence of indices
$(\ell_i)_{i\in\N}$ such that $|x(\ell_i)|>3\varepsilon$ for all $i$. Set
\[
A\ :=\ \{\ell_i:i\in\N\}\subset\N.
\]
For $n\in\N$ define
\[
E_n\ :=\ \bigl\{k\in A: |x_m(k)|\ge\varepsilon \text{ for some } m\le n\bigr\},
\]
and
\[
F_n\ :=\ \bigl\{k\in A: \sup_{m>n}|x_m(k)|\le 3\varepsilon\bigr\}.
\]

Since each $x_m\in c_0$, the set $\{k\in A:|x_m(k)|\ge\varepsilon\}$ is
finite, so $E_n$ is a finite union of finite sets and therefore finite.
Moreover, if $k\in F_n$, then $y(k):=\sup_{m\in\N}|x_m(k)|>3\varepsilon$
(because $k\in A$), and the condition $\sup_{m>n}|x_m(k)|\le 3\varepsilon$
forces the existence of some $m\le n$ with $|x_m(k)|>3\varepsilon$. Hence
\[
F_n\ \subset\ \bigcup_{m\le n}\bigl\{k\in A:|x_m(k)|>3\varepsilon\bigr\},
\]
so $F_n$ is finite as well. Thus for every $n$ the set
\[
A\setminus(E_n\cup F_n)
\]
is infinite.

We now construct strictly increasing sequences $(n_i)_{i\in\N}$ and
$(k_i)_{i\in\N}$ with $k_i\in A$ such that
\begin{equation}\label{eq:c0-big-jump}
|x_{n_{i+1}}(k_i)-x_{n_i}(k_i)|>\varepsilon\qquad(i\in\N).
\end{equation}
Set $n_1:=1$.  Since $A\setminus(E_{n_1}\cup F_{n_1})$ is infinite, we can
choose $k_1\in A\setminus(E_{n_1}\cup F_{n_1})$.  Then $k_1\notin E_{n_1}$
implies $|x_m(k_1)|<\varepsilon$ for all $m\le n_1$, so in particular
$|x_{n_1}(k_1)|<\varepsilon$.  As $k_1\notin F_{n_1}$, we have
\[
\sup_{m>n_1}|x_m(k_1)|>3\varepsilon,
\]
so there exists $n_2>n_1$ with $|x_{n_2}(k_1)|>3\varepsilon$.

Proceeding inductively, suppose $n_1<\cdots<n_i$ and $k_1<\cdots<k_i$ have
been chosen with $k_j\in A$ for $1\le j\le i$.  Since
$A\setminus(E_{n_i}\cup F_{n_i})$ is infinite and $\{k_1,\dots,k_i\}$ is
finite, we can choose
\[
k_{i+1}\in A\setminus(E_{n_i}\cup F_{n_i})
\quad\text{with}\quad k_{i+1}>k_i.
\]
Then $k_{i+1}\notin E_{n_i}$ implies
$|x_m(k_{i+1})|<\varepsilon$ for all $m\le n_i$, while
$k_{i+1}\notin F_{n_i}$ implies
\[
\sup_{m>n_i}|x_m(k_{i+1})|>3\varepsilon,
\]
so we may choose $n_{i+1}>n_i$ such that
$|x_{n_{i+1}}(k_{i+1})|>3\varepsilon$.  Hence
\[
|x_{n_{i+1}}(k_{i+1})-x_{n_i}(k_{i+1})|
\ge |x_{n_{i+1}}(k_{i+1})|-|x_{n_i}(k_{i+1})|
>3\varepsilon-\varepsilon = 2\varepsilon>\varepsilon,
\]
and \eqref{eq:c0-big-jump} holds (after relabelling indices if necessary).

Thus we have constructed strictly increasing sequences $(n_i)$ and $(k_i)$
with $k_i\in A$ such that \eqref{eq:c0-big-jump} holds for all $i$.  Let
$d_i:=x_{n_{i+1}}-x_{n_i}\in c_0$.  Then for each $i$ we have
\[
|d_i(k_i)|=|x_{n_{i+1}}(k_i)-x_{n_i}(k_i)|>\varepsilon.
\]

We claim that $(d_i)$ is not order bounded in $c_0$.  Indeed, suppose there
exists $0\le z\in c_0$ with $|d_i|\le z$ for all $i$.  Then
$z(k_i)\ge |d_i(k_i)|>\varepsilon$ for all $i$, which is impossible because
$z\in c_0$ can have $|z(k)|>\varepsilon$ only for finitely many $k$.  Thus
$(d_i)$ is not order bounded and hence cannot converge to $0$ in order.

However, the sequence $(x_n)$ is \Buo-Cauchy, so for every strictly
increasing sequence $(n_i)$ the differences $x_{n_{i+1}}-x_{n_i}$
must converge to $0$ in order, which contradicts the conclusion above.
Therefore our assumption $x\notin c_0$ was false, and so $x\in c_0$.

\emph{Case $\ell_\infty$.}
Suppose $(x_n)\subset\ell_\infty$ is \Buo-Cauchy and let $x$ be its
pointwise limit as above.  We first show that $(x_n)$ is norm bounded in
$\ell_\infty$.

Assume, towards a contradiction, that
\[
\sup_{n\in\N}\|x_n\|_\infty=\infty.
\]
We may then choose inductively a strictly increasing sequence
$(n_k)_{k\in\N}$ of integers such that
\[
\|x_{n_{k+1}}\|_\infty \ \ge\ \|x_{n_k}\|_\infty + (k+1)
\qquad(k\in\N).
\]
Set $d_k:=x_{n_{k+1}}-x_{n_k}\in\ell_\infty$.  By the triangle inequality,
\[
\|d_k\|_\infty
=\|x_{n_{k+1}}-x_{n_k}\|_\infty
\ge \|x_{n_{k+1}}\|_\infty-\|x_{n_k}\|_\infty
\ge k+1,
\]
so the sequence $(d_k)$ is not norm bounded.

On the other hand, since $(x_n)$ is \Buo-Cauchy and $(n_k)$ is a strictly
increasing sequence, the difference sequence $(d_k)$ must converge to $0$
in order in $\ell_\infty$.  Every order-null sequence is order bounded: if
$|d_k|\le z_m$ for all $k\ge N_m$ with a decreasing sequence
$(z_m)\downarrow0$, then
\[
y:=|d_1|\vee\cdots\vee|d_{N_1-1}|\vee z_1\in(\ell_\infty)_+
\]
satisfies $|d_k|\le y$ for all $k$.  In particular the set
$\{d_k:k\in\N\}$ is contained in the order interval $[-y,y]$, and every
order interval in a Banach lattice is norm bounded.  Hence $(d_k)$ is norm
bounded, contradicting $\|d_k\|_\infty\ge k+1$.

Therefore $(x_n)$ is norm bounded, and we may set
\[
M:=\sup_{n\in\N}\|x_n\|_\infty<\infty.
\]
For each fixed $k\in\N$ we then have
\[
|x(k)|=\lim_{n\to\infty}|x_n(k)|
\le \sup_{n\in\N}\|x_n\|_\infty=M,
\]
so $\|x\|_\infty\le M$ and hence $x\in\ell_\infty$.
\end{proof}

\begin{remark}
In the $\ell_\infty$ case, the above argument also shows that any
\Buo-Cauchy sequence $(x_n)\subset\ell_\infty$ is norm bounded:
$\sup_n\|x_n\|_\infty<\infty$.
\end{remark}

\begin{theorem}\label{thm:c0-complete}
The Banach lattices $c_0$ and $\ell_\infty$ are sequentially \Buo-complete.
\end{theorem}

\begin{proof}
We again treat the two cases separately.

\emph{Case $c_0$.}
Let $(x_n)\subset c_0$ be \Buo-Cauchy.  By the discussion preceding
Lemma~\ref{lem:pointwise_limit_in_c0} there is a pointwise limit
$x\in\R^\N$, and by Lemma~\ref{lem:pointwise_limit_in_c0} we have $x\in c_0$.

Since $c_0$ is a vector lattice of real-valued sequences, $\uo$-convergence
coincides with coordinatewise convergence.  Hence $x_n\xrightarrow{\uo}x$
in $c_0$.  To conclude that $x_n\xrightarrow{\Buo}x$ it suffices, by the
definition of \Buo-convergence, to show that $(x_n)$ is order bounded by a
single element of $c_0$.

Define
\[
y(k)\ :=\ \sup_{n\in\N}|x_n(k)|,\qquad k\in\N.
\]
Clearly $y\in\R^\N$ and $|x_n|\le y$ pointwise for all $n$.  We claim that
$y\in c_0$.  Suppose not.  Then there exist $\varepsilon>0$ and an infinite
subset $A\subset\N$ such that
\[
y(k)>3\varepsilon\qquad\text{for all }k\in A.
\]

For each $n\in\N$ define
\[
E_n\ :=\ \bigl\{k\in A: |x_m(k)|\ge\varepsilon \text{ for some } m\le n\bigr\},
\qquad
F_n\ :=\ \bigl\{k\in A: \sup_{m>n}|x_m(k)|\le 3\varepsilon\bigr\}.
\]
As in the proof of Lemma~\ref{lem:pointwise_limit_in_c0}, both $E_n$ and
$F_n$ are finite for every $n$, so $A\setminus(E_n\cup F_n)$ is infinite.

We now construct strictly increasing sequences $(n_i)_{i\in\N}$ and $(k_i)_{i\in\N}$
with $k_i\in A$ such that
\begin{equation}\label{eq:c0-big-jump-2}
|x_{n_{i+1}}(k_i)-x_{n_i}(k_i)|>\varepsilon\qquad(i\in\N).
\end{equation}

Set $n_1:=1$. Choose $k_1\in A\setminus(E_{n_1}\cup F_{n_1})$.
Then $|x_{n_1}(k_1)|<\varepsilon$ and, since $k_1\notin F_{n_1}$,
there exists $n_2>n_1$ with $|x_{n_2}(k_1)|>3\varepsilon$, so
\eqref{eq:c0-big-jump-2} holds for $i=1$.

Assume $n_1<\cdots<n_i$ and $k_1<\cdots<k_i$ have been chosen and
$k_i\in A\setminus(E_{n_i}\cup F_{n_i})$.
Since $k_i\notin F_{n_i}$, there exists $n_{i+1}>n_i$ with $|x_{n_{i+1}}(k_i)|>3\varepsilon$.
As $k_i\notin E_{n_i}$ we have $|x_{n_i}(k_i)|<\varepsilon$, hence
\[
|x_{n_{i+1}}(k_i)-x_{n_i}(k_i)|
\ge |x_{n_{i+1}}(k_i)|-|x_{n_i}(k_i)|
>3\varepsilon-\varepsilon=2\varepsilon>\varepsilon,
\]
which gives \eqref{eq:c0-big-jump-2}.

Finally, since $A\setminus(E_{n_{i+1}}\cup F_{n_{i+1}})$ is infinite, we may pick
$k_{i+1}\in A\setminus(E_{n_{i+1}}\cup F_{n_{i+1}})$ with $k_{i+1}>k_i$.
This completes the construction.

Set $d_i:=x_{n_{i+1}}-x_{n_i}\in c_0$.  Then
\[
|d_i(k_i)|=|x_{n_{i+1}}(k_i)-x_{n_i}(k_i)|>\varepsilon\qquad(i\in\N).
\]
If there existed $0\le z\in c_0$ with $|d_i|\le z$ for all $i$, then
$z(k_i)\ge |d_i(k_i)|>\varepsilon$ for all $i$, contradicting the fact that
$z\in c_0$ has only finitely many coordinates exceeding $\varepsilon$ in
absolute value.  Thus the sequence $(d_i)$ is not order bounded in $c_0$,
and hence cannot converge to $0$ in order, which contradicts the
\Buo-Cauchy property of $(x_n)$ applied to the subsequence $(x_{n_i})_i$.

Therefore our assumption $y\notin c_0$ was false, and we must have
$y\in c_0$.  Since $|x_n|\le y$ for all $n$ and $x_n\xrightarrow{\uo}x$ in
$c_0$, the definition of \Buo-convergence yields $x_n\xrightarrow{\Buo}x$.
Thus $c_0$ is sequentially \Buo-complete.

\medskip

\emph{Case $\ell_\infty$.}
Let $(x_n)\subset\ell_\infty$ be \Buo-Cauchy and let $x\in\ell_\infty$ be
its pointwise limit, as provided by Lemma~\ref{lem:pointwise_limit_in_c0}.
By the same lemma and its proof, $(x_n)$ is norm bounded:
\[
M:=\sup_{n\in\N}\|x_n\|_\infty<\infty.
\]
Since $\ell_\infty$ is a vector lattice of real-valued sequences,
coordinatewise convergence coincides with $\uo$-convergence, so
$x_n\xrightarrow{\uo}x$ in $\ell_\infty$.

Define
\[
y(k)\ :=\ \sup_{n\in\N}|x_n(k)|,\qquad k\in\N.
\]
Then $|x_n|\le y$ pointwise for all $n$, and
\[
\|y\|_\infty
=\sup_{k\in\N}y(k)
=\sup_{k\in\N}\sup_{n\in\N}|x_n(k)|
=\sup_{n\in\N}\|x_n\|_\infty
\le M<\infty,
\]
so $y\in(\ell_\infty)_+$.  Thus the sequence $(x_n)$ is order bounded by
$y$ in $\ell_\infty$.  Since we already know $x_n\xrightarrow{\uo}x$, the
definition of \Buo-convergence implies $x_n\xrightarrow{\Buo}x$.

Hence $\ell_\infty$ is sequentially \Buo-complete as well.
\end{proof}

\begin{remark}\label{rem:c0-linfty-open}
Combining Lemma~\ref{lem:pointwise_limit_in_c0} with
Theorem~\ref{thm:c0-complete}, we see that every \Buo-Cauchy sequence in
$c_0$ or in $\ell_\infty$ has a pointwise limit in the same space and in
fact converges to this limit in order (equivalently, in the \Buo\ sense).
In particular, both $c_0$ and $\ell_\infty$ enjoy full sequential
\Buo-completeness.
\end{remark}

Without resolving sequential \Buo-completeness for $\ell_p$-spaces in the reflexive range, we record the following lemma, which may prove useful.

\begin{lemma}\label{lem:pointwise_limit_in_ellp}
Let $1\le p<\infty$. If $(x_n)\subset \ell_p$ is \Buo-Cauchy and $x\in\R^\N$ is its coordinatewise limit, then $x\in \ell_p$.
\end{lemma}

\begin{proof}
Suppose, for a contradiction, that $x\notin \ell_p$. Then $\sum_{j=1}^\infty |x(j)|^p=\infty$, hence for every $K\in\N$
the tail sum $\sum_{j\ge K}|x(j)|^p$ is infinite.

We construct strictly increasing indices $(n_i)$ and pairwise disjoint finite intervals
\[
I_i=[k_i,\ell_i)\cap\N \qquad (i\in\N)
\]
such that
\begin{equation}\label{eq:ellp-big-block}
\bigl\|(x_{n_{i+1}}-x_{n_i})\,\mathds1_{I_i}\bigr\|_p>1\qquad(i\in\N).
\end{equation}

Set $n_1:=1$ and choose $k_1$ so large that $\|x_{n_1}\mathds1_{[k_1,\infty)}\|_p<\tfrac14$.
Since $\sum_{j\ge k_1}|x(j)|^p=\infty$, choose $\ell_1>k_1$ so that
$\|x\mathds1_{I_1}\|_p>2$ for $I_1=[k_1,\ell_1)$.
Because $I_1$ is finite and $x_n\to x$ coordinatewise, choose $n_2>n_1$ with
$\|(x_{n_2}-x)\mathds1_{I_1}\|_p<\tfrac14$. Then
\[
\|(x_{n_2}-x_{n_1})\mathds1_{I_1}\|_p
\ge \|x\mathds1_{I_1}\|_p-\|(x_{n_2}-x)\mathds1_{I_1}\|_p-\|x_{n_1}\mathds1_{I_1}\|_p
>2-\tfrac14-\tfrac14>1,
\]
so \eqref{eq:ellp-big-block} holds for $i=1$.

Assume $n_i,k_i,\ell_i$ have been chosen. Choose $k_{i+1}>\ell_i$ so large that
$\|x_{n_{i+1}}\mathds1_{[k_{i+1},\infty)}\|_p<\tfrac14$.
Since $\sum_{j\ge k_{i+1}}|x(j)|^p=\infty$, choose $\ell_{i+1}>k_{i+1}$ so that
$\|x\mathds1_{I_{i+1}}\|_p>2$ for $I_{i+1}=[k_{i+1},\ell_{i+1})$.
Finally choose $n_{i+2}>n_{i+1}$ so that $\|(x_{n_{i+2}}-x)\mathds1_{I_{i+1}}\|_p<\tfrac14$.
The same estimate as above gives \eqref{eq:ellp-big-block} for $i+1$.

Now set $d_i:=x_{n_{i+1}}-x_{n_i}$. If $(d_i)$ were order bounded in $\ell_p$, there would exist $0\le y\in\ell_p$
with $|d_i|\le y$ for all $i$, hence $\|y\mathds1_{I_i}\|_p\ge \||d_i|\mathds1_{I_i}\|_p>1$ for all $i$.
Since the intervals $I_i$ are disjoint,
\[
\|y\|_p^p \ \ge\ \sum_{i=1}^\infty \|y\mathds1_{I_i}\|_p^p \ >\ \sum_{i=1}^\infty 1 \ =\ \infty,
\]
contradicting $y\in\ell_p$. Thus $(d_i)$ is not order bounded, so it cannot converge to $0$ in order.

But $(x_n)$ is \Buo-Cauchy, so for the subsequence $(x_{n_i})$ the differences $x_{n_{i+1}}-x_{n_i}=d_i$
must converge to $0$ in order, a contradiction. Therefore $x\in\ell_p$.
\end{proof}

\section{Lipschitz function spaces}\label{sec:LipX}

\begin{definition}\label{def:sep-radius}
Let $(X,d)$ be a metric space. For each $x\in X$, define the \emph{isolation radius}
\[
d(x):=\inf\{d(x,y):y\in X,\ y\neq x\}\in[0,\infty].
\]
The \emph{uniform discreteness constant} of $X$ is
\[
\delta(X):=\inf_{x\in X} d(x)\in[0,\infty].
\]
We say $X$ is \emph{uniformly discrete} if $\delta(X)>0$.

\smallskip
We write $\mathrm{Lip}_b(X)$ for the vector lattice of all \emph{bounded}
Lipschitz functions $f:X\to\R$, ordered pointwise.
\end{definition}

On $\R^X$ (or $L_0(\mu)$) unbounded order convergence agrees with pointwise (resp.\ a.e.) convergence,
but in sublattices such as $\mathrm{Lip}_b(X)$ this identification can fail; see Remark~\ref{rem:uo_not_pointwise}.
In what follows we work with the pointwise order on $\mathrm{Lip}_b(X)$.

\begin{lemma}\label{lem:UC-approx-by-Lip}
Let $(X,d)$ be a metric space and let $g:X\to\R$ be bounded and uniformly continuous.
For $n\in\N$, define the \emph{inf-convolution}
\[
g_n(x):=\inf_{y\in X}\bigl(g(y)+n\,d(x,y)\bigr)\qquad(x\in X).
\]
Then:
\begin{enumerate}[\upshape(i)]
\item each $g_n$ is bounded and $n$-Lipschitz;
\item $g_n\le g$ pointwise;
\item $\|g_n-g\|_\infty\to0$ as $n\to\infty$;
\item $(g_n)$ is increasing: if $m\ge n$, then $g_n\le g_m$ pointwise.
\end{enumerate}
\end{lemma}

\begin{proof}
\emph{(i)} Fix $n\in\N$ and let $x,x'\in X$. For any $y\in X$, the triangle inequality gives
\[
g(y)+n\,d(x,y)\le g(y)+n\,d(x',y)+n\,d(x,x').
\]
Taking the infimum over $y$ on the right-hand side yields $g_n(x)\le g_n(x')+n\,d(x,x')$.
By symmetry, $g_n(x')\le g_n(x)+n\,d(x,x')$, hence $|g_n(x)-g_n(x')|\le n\,d(x,x')$.

For boundedness, taking $y=x$ shows $g_n(x)\le g(x)$, so $\sup_x g_n(x)\le \|g\|_\infty$.
Also, we have $g(y)\ge -\|g\|_\infty$ and $d(x,y)\ge0$, hence $g_n(x)\ge -\|g\|_\infty$.
\smallskip

\emph{(ii)} Taking $y=x$ in the definition of $g_n(x)$ gives
$g_n(x)\le g(x)+n\cdot 0=g(x)$.

\smallskip
\emph{(iii)} Define the modulus of uniform continuity
\[
\omega(t):=\sup\{|g(x)-g(y)|:d(x,y)\le t\}\qquad(t\ge0).
\]
Then $\omega(t)\to0$ as $t\downarrow0$ and $\omega(t)\le 2\|g\|_\infty$ for all $t\ge0$.

For any $x,y\in X$,
\[
g(y)\ge g(x)-|g(x)-g(y)|\ge g(x)-\omega(d(x,y)),
\]
so
\[
g(y)+n\,d(x,y)\ge g(x)-\omega(d(x,y))+n\,d(x,y)
= g(x)-\bigl(\omega(d(x,y))-n\,d(x,y)\bigr).
\]
Taking the infimum over $y\in X$ yields
\[
g_n(x)\ge g(x)-\sup_{t\ge0}\bigl(\omega(t)-nt\bigr).
\]
Set $\alpha_n:=\sup_{t\ge0}(\omega(t)-nt)\ge0$. Together with $g_n\le g$ we get
\[
0\le g(x)-g_n(x)\le \alpha_n\qquad(x\in X),
\]
hence $\|g-g_n\|_\infty\le \alpha_n$.

It remains to show $\alpha_n\to0$. Fix $\varepsilon>0$ and choose $t_0>0$ so that $\omega(t_0)<\varepsilon$.
For $t\in[0,t_0]$ we have $\omega(t)-nt\le \omega(t)\le \omega(t_0)<\varepsilon$.
For $t>t_0$, using $\omega(t)\le 2\|g\|_\infty$ gives
\[
\omega(t)-nt\le 2\|g\|_\infty-nt_0.
\]
If $n>2\|g\|_\infty/t_0$, then $2\|g\|_\infty-nt_0<0<\varepsilon$.
Thus $\omega(t)-nt<\varepsilon$ for all $t\ge0$, so $\alpha_n\le\varepsilon$.
Since $\varepsilon>0$ is arbitrary, $\alpha_n\to0$.
\smallskip

\emph{(iv)} If $m\ge n$, then for every $y\in X$ we have
$g(y)+n\,d(x,y)\le g(y)+m\,d(x,y)$.
Taking the infimum over $y$ yields $g_n(x)\le g_m(x)$.
\end{proof}

\begin{lemma}\label{lem:many-close-pairs}
Let $(X,d)$ be a metric space with $\delta(X)=0$ and $d(x)>0$ for all $x\in X$
(i.e.,\ $X$ is discrete but not uniformly discrete). Then for every $\varepsilon>0$
and every finite set $F\subset X$, there exist distinct points $a,b\in X\setminus F$
with $d(a,b)<\varepsilon$.
\end{lemma}

\begin{proof}
Fix $\varepsilon>0$ and a finite set $F\subset X$.
If $F\neq\varnothing$, set $\eta:=\min_{x\in F} d(x)>0$; if $F=\varnothing$, set $\eta:=\varepsilon$.
Let
\[
\theta:=\frac14\min\{\varepsilon,\eta\}>0.
\]
Since $\delta(X)=0$, there exists $a\in X$ with $d(a)<\theta$.
Then $a\notin F$, because $d(x)\ge\eta>\theta$ for all $x\in F$.

By the definition of $d(a)$, choose $b\neq a$ such that
\[
d(a,b)<d(a)+\theta<2\theta\le \frac12\min\{\varepsilon,\eta\}.
\]
In particular $d(a,b)<\varepsilon$. We claim that $b\notin F$.
Indeed, if $b\in F$ then
\[
d(b)\le d(b,a)=d(a,b)<\eta/2,
\]
contradicting $d(b)\ge \eta$ by definition of $\eta$.
Thus $a,b\in X\setminus F$ are distinct and satisfy $d(a,b)<\varepsilon$.
\end{proof}

\begin{proposition}\label{prop:Lipb-uniformly-discrete}
Let $(X,d)$ be a metric space and let $\mathrm{Lip}_b(X)$ be the vector lattice of bounded
Lipschitz functions $X\to\R$, ordered pointwise. Then $\mathrm{Lip}_b(X)$ is sequentially
\Buo-complete if and only if $\delta(X)>0$.
\end{proposition}

\begin{proof}
\noindent For necessity, assume $\delta(X)=0$. We construct a \Buo-Cauchy sequence in
$\mathrm{Lip}_b(X)$ with no order limit.

\smallskip\noindent
We first produce a closed set $A\subset X$ and a sequence $(b_n)\subset X\setminus A$ such that
\[
t_n:=\mathrm{dist}(b_n,A):=\inf_{a\in A}d(b_n,a)\in(0,1)\quad\text{and}\quad t_n\downarrow0.
\]

\smallskip
\emph{Case A: there exists $x_0\in X$ with $d(x_0)=0$.}
Then $x_0$ is non-isolated, so there exist $b_n\neq x_0$ with $d(b_n,x_0)\to0$.
Let $A:=\{x_0\}$. Then $t_n=\mathrm{dist}(b_n,A)=d(b_n,x_0)\downarrow0$; passing to a tail, assume $t_n<1$ for all $n$.

\smallskip
\emph{Case B: $d(x)>0$ for all $x\in X$.}
Then $X$ is discrete, and since $\delta(X)=0$ it is not uniformly discrete.
We build inductively pairs $(a_n,b_n)$ of distinct points such that
\[
d(a_n,b_n)<\frac1n
\quad\text{and}\quad
\{a_1,\dots,a_n\}\cap\{b_1,\dots,b_n\}=\varnothing.
\]
Choose any distinct $a_1,b_1$ with $d(a_1,b_1)<1$.
Assuming $(a_1,b_1),\dots,(a_{n-1},b_{n-1})$ chosen, let
\[
F:=\{a_1,\dots,a_{n-1},b_1,\dots,b_{n-1}\}.
\]
By Lemma~\ref{lem:many-close-pairs}, there exist distinct $a_n,b_n\in X\setminus F$ with $d(a_n,b_n)<1/n$.
Set $A:=\{a_n:n\in\N\}$. Then $b_n\notin A$ and
\[
t_n=\mathrm{dist}(b_n,A)\le d(b_n,a_n)<\frac1n<1,
\]
so $t_n\downarrow0$. Moreover $t_n>0$ since $b_n\notin A$ and $d(b_n)>0$.

\smallskip\noindent
In both cases, define the bounded uniformly continuous function
\[
g(x):=\sqrt{\mathrm{dist}(x,A)}\wedge 1\qquad(x\in X).
\]
(Here $x\mapsto \mathrm{dist}(x,A)$ is $1$-Lipschitz and $t\mapsto \sqrt{t}\wedge 1$
is uniformly continuous on $[0,\infty)$.)

\smallskip\noindent
\emph{$g$ is not Lipschitz.}
For each $n$, choose $a'_n\in A$ such that $d(b_n,a'_n)<2t_n$ (possible by the definition of $t_n$).
Then $g(a'_n)=0$ and $g(b_n)=\sqrt{t_n}$, hence
\[
\frac{|g(b_n)-g(a'_n)|}{d(b_n,a'_n)}
>\frac{\sqrt{t_n}}{2t_n}
=\frac{1}{2\sqrt{t_n}}
\longrightarrow\infty,
\]
so $g\notin\mathrm{Lip}_b(X)$.

\smallskip\noindent
\emph{Approximation by bounded Lipschitz functions.}
By Lemma~\ref{lem:UC-approx-by-Lip}, the inf-convolutions
\[
g_n(x):=\inf_{y\in X}\bigl(g(y)+n\,d(x,y)\bigr)
\]
belong to $\mathrm{Lip}_b(X)$ and satisfy $\|g_n-g\|_\infty\to0$.

\smallskip\noindent
\emph{$(g_n)$ is \Buo-Cauchy.}
Let $(n_k)$ be strictly increasing. Then
\[
\|g_{n_{k+1}}-g_{n_k}\|_\infty
\le \|g_{n_{k+1}}-g\|_\infty+\|g-g_{n_k}\|_\infty
\longrightarrow0.
\]
Define $\varepsilon_m:=\sup_{k\ge m}\|g_{n_{k+1}}-g_{n_k}\|_\infty$. Then $\varepsilon_m\downarrow0$, and for all $k\ge m$,
\[
|g_{n_{k+1}}-g_{n_k}|\le \varepsilon_m\,\mathds1_X,
\]
where $\mathds1_X$ is the constant $1$ function. Since $\varepsilon_m\mathds1_X\downarrow0$ in $\mathrm{Lip}_b(X)$,
it follows that $g_{n_{k+1}}-g_{n_k}\xrightarrow{\mathrm{o}}0$.
Thus $(g_n)$ is \Buo-Cauchy.

\smallskip\noindent
\emph{$(g_n)$ has no order limit in $\mathrm{Lip}_b(X)$.}
Suppose towards a contradiction that $g_n\xrightarrow{\mathrm{o}}h$ in
$\mathrm{Lip}_b(X)$. By Lemma~\ref{lem:UC-approx-by-Lip}\textup{(iv)} the
sequence $(g_n)$ is increasing, hence Lemma~\ref{lem:monotone_order_limit}
yields that $h$ is the supremum of $\{g_n:n\in\N\}$ in
$\mathrm{Lip}_b(X)$.

For each $m\in\N$, choose $N_m\in\N$ such that
$\|g-g_{N_m}\|_\infty<\tfrac1m$, and define
\[
u_m:=g_{N_m}+\frac1m\,\mathds{1}_X\in \mathrm{Lip}_b(X).
\]
Since $g_n\le g$ for all $n$ and $g\le u_m$ pointwise, each $u_m$ is an
upper bound of the set $\{g_n:n\in\N\}$ in $\mathrm{Lip}_b(X)$. By minimality
of the supremum, we therefore have
\[
h\le u_m\qquad(m\in\N).
\]
Fix $x\in X$. Since $g_{N_m}(x)\to g(x)$ and $m^{-1}\to 0$, it follows that
$u_m(x)\to g(x)$, hence $h(x)\le g(x)$. On the other hand, $h$ is itself an
upper bound of $\{g_n:n\in\N\}$, while $g_n(x)\uparrow g(x)$ pointwise, so
$h(x)\ge g(x)$. Thus $h(x)=g(x)$ for all $x\in X$, i.e.\ $h=g$.
This contradicts $g\notin \mathrm{Lip}_b(X)$.
Therefore $\mathrm{Lip}_b(X)$ is not sequentially \Buo-complete when
$\delta(X)=0$.

\medskip
\noindent For sufficiency, assume $\delta:=\delta(X)>0$.
Let $f:X\to\R$ be bounded. For distinct $x,y\in X$ we have $d(x,y)\ge\delta$, hence
\[
\frac{|f(x)-f(y)|}{d(x,y)}\le \frac{2\|f\|_\infty}{\delta}.
\]
Thus every bounded function is Lipschitz, and consequently
\[
\mathrm{Lip}_b(X)=\ell_\infty(X)
\]
as vector lattices (both are all bounded functions on $X$ with the pointwise order).

It therefore suffices to show that $\ell_\infty(X)$ is sequentially \Buo-complete.
Let $(f_n)\subset\ell_\infty(X)$ be \Buo-Cauchy.

\smallskip\noindent
\emph{Pointwise convergence.}
Fix $x\in X$. For any strictly increasing $(n_k)$, the differences
$f_{n_{k+1}}-f_{n_k}$ converge to $0$ in order in $\ell_\infty(X)$, hence in particular
\[
f_{n_{k+1}}(x)-f_{n_k}(x)\to0\quad\text{in }\R.
\]
Thus $(f_n(x))_n$ is a Cauchy sequence in $\R$, hence convergent.
Define $f(x):=\lim_{n\to\infty}f_n(x)$.

\smallskip\noindent
\emph{Uniform boundedness.}
We claim $\sup_n\|f_n\|_\infty<\infty$.
Suppose not. Choose a strictly increasing $(n_k)$ with
\[
\|f_{n_{k+1}}\|_\infty\ge \|f_{n_k}\|_\infty+(k+1)\qquad(k\in\N).
\]
Set $d_k:=f_{n_{k+1}}-f_{n_k}$. Then
\[
\|d_k\|_\infty\ge \|f_{n_{k+1}}\|_\infty-\|f_{n_k}\|_\infty\ge k+1,
\]
so $(d_k)$ is not norm bounded.
However $(d_k)$ converges to $0$ in order by \Buo-Cauchyness, hence it is order bounded:
there exists $0\le y\in\ell_\infty(X)$ with $|d_k|\le y$ for all $k$.
Then $\|d_k\|_\infty\le\|y\|_\infty$ for all $k$, a contradiction.
Thus $M:=\sup_n\|f_n\|_\infty<\infty$, and in particular $|f(x)|\le M$ for all $x$, so $f\in\ell_\infty(X)$.

\smallskip\noindent
\emph{Order convergence.}
For $m\in\N$ define
\[
z_m(x):=\sup_{n\ge m}|f_n(x)-f(x)|,\qquad x\in X.
\]
Then $0\le z_{m+1}\le z_m$ pointwise, and $z_m(x)\downarrow0$ for each $x$ since $f_n(x)\to f(x)$.
Also $z_m(x)\le 2M$ for all $x$, hence $z_m\in\ell_\infty(X)$.
Finally, for every $n\ge m$ we have $|f_n-f|\le z_m$, so $f_n\xrightarrow{\mathrm{o}}f$ in $\ell_\infty(X)$.
By Proposition~\ref{prop:Buo=order}, this is equivalent to $f_n\xrightarrow{\Buo}f$.
Therefore $\ell_\infty(X)$ (and hence $\mathrm{Lip}_b(X)$) is sequentially \Buo-complete.
\end{proof}

\begin{remark}\label{rem:Lipb-characterisation}
Proposition~\ref{prop:Lipb-uniformly-discrete} shows that sequential \Buo-completeness of $\mathrm{Lip}_b(X)$
depends only on the metric geometry of $X$: it holds exactly when $X$ is uniformly discrete.
In particular:
\begin{enumerate}[\upshape(i)]
\item if $X$ has a non-isolated point, then $\delta(X)=0$ and $\mathrm{Lip}_b(X)$ is not sequentially \Buo-complete;
\item if $X$ is discrete but not uniformly discrete (so $d(x)>0$ for all $x$ but $\delta(X)=0$), then $\mathrm{Lip}_b(X)$ is not sequentially \Buo-complete;
\item if $\delta(X)>0$, then $\mathrm{Lip}_b(X)=\ell_\infty(X)$ and $\mathrm{Lip}_b(X)$ is sequentially \Buo-complete.
\end{enumerate}
\end{remark}

The above reasoning was inspired by a characterisation of compact metric spaces provided by \cite{HH}.

\section{Open problems}\label{sec:open-problems}
\begin{question}\label{q:norm-bounded-buo}
Characterise those Banach lattices for which every norm bounded \Buo-Cauchy
sequence is $\uo$-convergent.  How does this property relate to the bounded
uo-completeness studied in \cite{GaoTroitskyXanthosBuoComp}?
\end{question}
\begin{question}\label{q:c0-ellp-strong}
Do \Buo-Cauchy sequences converge in
\begin{itemize}
    \item $\ell_p$ for $1<p<\infty$?
    \item the Morrey spaces $\mathcal M^{p,\kappa}(\T)$ and in Besov/Triebel--Lizorkin spaces $B^s_{p,q}(\T),F^s_{p,q}(\T)$ in the usual parameter ranges?
\end{itemize}
\end{question}

\begin{question}\label{q:MO-scope}
For Musielak--Orlicz spaces, can one characterise those $\Psi$ for which \Buo\ completeness fails? What is the optimal role of $\Delta_2$ and of the underlying measure?
\end{question}

\section*{Acknowledgements}
We are indebted to the Library of the Institute of Mathematics of the Polish Academy of Sciences in Warsaw for providing us a copy of Bourgain's paper \cite{Bourgain81}.


\begin{thebibliography}{99}

\bibitem{AbramovichSirotkin2005}
Y.~A. Abramovich and G.~Sirotkin,
\newblock On order convergence of nets,
\newblock \emph{Positivity} \textbf{9} (2005), no.~2, 287--292.

\bibitem{AlbiacKalton2006}
F.~Albiac and N.~J. Kalton,
\newblock \emph{Topics in Banach Space Theory}, 2nd ed.,
\newblock Springer, 2016.

\bibitem{AvilesRodriguezTradacete2017}
A.~Avil\'es, J.~Rodr\'iguez, and P.~Tradacete,
\newblock The free Banach lattice generated by a Banach space,
\newblock \emph{J. Funct. Anal.} \textbf{274} (2018), no. 10, 2955--2977.
% See also arXiv:1706.08147.

\bibitem{Bourgain81}
J.~Bourgain,
\newblock Non-completeness of some convergence on $\ell_1$,
\newblock \emph{Colloq. Math.} \textbf{44} (1981), no. 1, 175--178.

\bibitem{DashiellHagerHenriksen1980}
F.~Dashiell, A.~Hager, and M.~Henriksen,
\newblock Order-{C}auchy completions of rings and vector lattices of continuous functions,
\newblock \emph{Canad. J. Math.} \textbf{32} (1980), no.~3, 657--685.


\bibitem{GaoLeungXanthos2018}
N.~Gao, D.~Leung, and F.~Xanthos,
\newblock Duality for uo-convergence and applications,
\newblock \emph{Positivity} \textbf{22} (2018), no.~3, 711--725.

\bibitem{GaoXanthos2014}
N.~Gao and F.~Xanthos,
\newblock Unbounded order convergence and application to martingales without probability,
\newblock \emph{J. Math. Anal. Appl.} \textbf{415} (2014), 931--947.

\bibitem{GaoTroitskyXanthos2017}
N.~Gao, V.~G. Troitsky, and F.~Xanthos,
\newblock Uo-convergence and its applications to {C}es\`aro means in {B}anach lattices,
\newblock \emph{Israel J. Math.} \textbf{220} (2017), no.~2, 649--689.

\bibitem{GaoTroitskyXanthosBuoComp}
N.~Gao, V.~G. Troitsky, and F.~Xanthos,
\newblock On the C-property and bounded uo-completeness of Banach lattices,
\newblock \emph{J. Math. Anal. Appl.} \textbf{463} (2018), 310--319.

\bibitem{HH} H.~Hueber,
\newblock On Uniform Continuity and Compactness in Metric Spaces, 
\newblock \emph{Amer. Math. Monthly} \textbf{88} (1981), no. 3, 204--205.

\bibitem{KandicTaylor2018}
M.~Kandi\'c and M.~A. Taylor,
\newblock Metrizability of minimal and unbounded topologies,
\newblock \emph{J. Math. Anal. Appl.} \textbf{466} (2018), no.~1, 144--159.

\bibitem{LindenstraussTzafriri1979}
J. Lindenstrauss and L. Tzafriri,
\newblock \emph{Classical Banach Spaces II: Function Spaces},
\newblock Springer-Verlag, 1979.


\bibitem{OikhbergTaylorTradaceteTroitsky2024}
T.~Oikhberg, M.~A. Taylor, P.~Tradacete, and V.~G. Troitsky,
\newblock Free Banach lattices,
\newblock \emph{J. Eur. Math. Soc. (JEMS)} (online first, 2024); see also \href{https://arxiv.org/abs/2210.00614}{arXiv:2210.00614}.

\bibitem{Taylor2019}
M.~A. Taylor,
\newblock Unbounded topologies and uo-convergence in locally solid vector lattices,
\newblock \emph{J. Math. Anal. Appl.} \textbf{472} (2019), no.~1, 981--1000.

\bibitem{TaylorThesis2018}
M.~A. Taylor,
\newblock \emph{Unbounded Convergences in Vector Lattices},
\newblock Ph.D. thesis, University of Alberta, 2018.


\end{thebibliography}
\end{document}